\documentclass[11pt,reqno]{article}
\date{}
\usepackage{amstext,amsmath,amssymb,amsthm, bm}
\usepackage{subfig}
\usepackage[english]{babel}
\usepackage{graphicx}
\usepackage{tabularx}
\usepackage{enumerate, hyperref}
\usepackage{tikz}
\usepackage{ifthen}
\usetikzlibrary{shapes.geometric,calc}
\def\tri
{
	\ifnum\count5=\topDepth
	\draw 
	(\the\numexpr\count7+2\relax,\the\count8) -| 
	(\the\count7,\the\numexpr\count8+2\relax) -- 
	cycle;
	\else
	\begingroup
	\advance\count5 by 1
	\multiply\count6 by 2
	\tri
	\begingroup
	\advance\count7 by\count6
	\tri
	\endgroup
	\begingroup
	\advance\count8 by\count6
	\tri
	\endgroup
	\endgroup
	\fi
}
\pgfmathsetmacro{\mySqrt}{sqrt(.75)}

\def\({\Big(}
\def\){\Big)}
\newcommand{\R}{{ \mathbb  R  }}

\newcommand{\Z}{  \mathbb Z }
\newcommand{\N}{  \mathbb N }
 \newcommand{\ov}{  \overline}

 \newcommand{\Q}{    T}

\newtheorem{Thm}{Theorem}[section]
\newtheorem{Lemma}[Thm]{Lemma}

\newcommand{\dsize}{\displaystyle}
\renewcommand{\cal}{\mathcal}
\numberwithin{equation}{section}

\title  {{\Large   On  the Sierpiński triangle and its generalizations}}

\author{L. De Carli\thanks{Dept. Mathematics and statistics, Florida International University,   Miami, FL 33199, USA. 
 email:{decarlil@fiu.edu}} 
\and A. Echezabal \thanks{Dept. Mathematics and statistics, Florida International University,   Miami, FL 33199, USA. email:{ ancheza@fiu.edu}
}
\and I. Morell \thanks{Dept. Mathematics and statistics, Florida International University,   Miami, FL 33199, USA.  email:{imore040@fiu.edu  }}}

\begin{document}
	
	%\subjclass[2010] 

\maketitle


\begin{abstract}
	By examining arithmetic operations between numbers expressed in base $m\ge 2$, we uncover new families of fractal sets in the plane that include the  classical Sierpiński triangle as a special case.

\end{abstract}


\section{Introduction}
Fractals are fascinating objects in mathematics.  
They are figures   where each part, when magnified, resembles the whole (a property known as {\it self-similarity})  and  defy  the traditional notion  of dimension. 
Fractals  also appear in  fields as diverse as dynamical systems, number theory, computer graphics, and even biology and finance  \cite{BP, F, mandelbrot, P, W}.

One of the most famous and historically relevant examples of fractals is the  Sierpiński triangle.   It is named after the Polish mathematician Waclaw Sierpiński (1882 -1969)    but it appeared as a decorative pattern many centuries before his work   \cite{Si, C, MW}.
\begin{figure}[h]
	\includegraphics[width=0.4\textwidth]{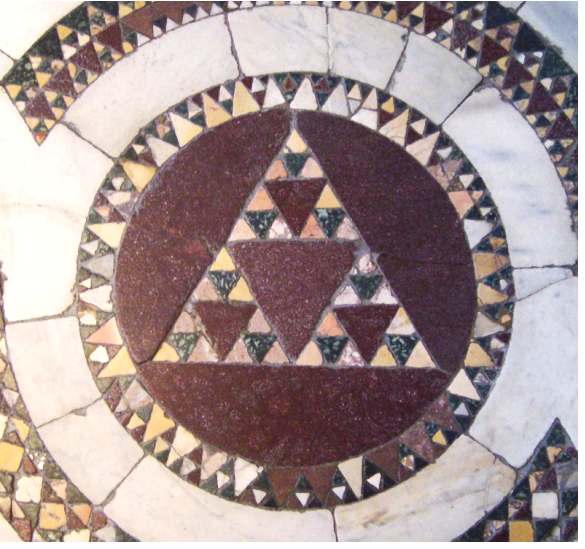}
	\quad \def\topDepth{5}
	\begin{tikzpicture}[scale=2.7*.5^\topDepth,cm={1,0,.5,\mySqrt,(0,0)}]
		\tri
	\end{tikzpicture}
	\caption{\small{ Details of a stone mosaic of the 11th century) in the San Clemente basilica, Rome (Foto: Carlini)   }}
	\label{fig:sierpinski}
\end{figure}

The Sierpiński triangle   can be constructed through a recursive process of removing smaller and smaller   triangles from an initial triangle (see Figure~\ref{fig:sierpinski}). It can also be characterized as the set of points $(x, y)$ in the square $Q=[0,1] \times [0,1]$ such that the base 2 representation  of $x$ and $y$ never have a $1$ in the same digit position; see, for instance, \cite[Exercise~1.2.4]{E}.

\medskip
In this paper we use the construction of the Sierpiński triangle   as a model  for  the construction of a broader class of planar fractals which, for integers $m \geq 2$, can be characterized in terms of the base  $m$ representations of  their points.  The connection  between digit patterns  of numbers in base $m$  and fractal structures is largely   known (see, e.g., \cite[Chapter~10]{F}) but, to the best of our knowledge, the  fractals that we construct in this paper  are not in the literature.
We also produce  fractals   whose points   can be characterized  in terms of   their representation  in a modified (or: \emph{balanced}) versions of the standard base $m$.

%Then, we  construct new families of fractals based on the properties of   the modified base $m$ representations of  their points

\medskip
Our paper is organized as follows. In Section 2 we review some basic properties of   numbers in different  bases; in section 3  and 4 we define  two new classes of fractals  that we term {\it  triangle fractals } and {\it hexagon fractals},   and we evaluate their  Minkowski dimension in section 5.

\section{ Numbers  in different bases}

Our   number system is in base ten: each digit of a number  represents a multiple of a power of $10$, and the value of the number is  the sum of  these weighted terms. 
More generally,   any non-negative integer $x$ can be uniquely expressed as a finite sum
\begin{equation}\label{e-base m}
	x = x_0 + x_1 m + x_2 m^2 + \cdots + x_k m^k,
\end{equation}
where   $m \ge  2$ is a given integer, and  the  coefficients $x_i $ range from $0$ to $m-1$. This representation allows   to interpret   $x$ as a polynomial in $m$ with coefficients in the ring $\mathbb{Z}_m := \mathbb{Z}/m\mathbb{Z}$ of congruence classes modulo $m$.
We represent $\mathbb{Z}_m $  by the set of  integers  $ \{0, 1, \ldots, m - 1\}$.

The  base $m = 2$ is especially prominent in applications and it allows to represent   numbers   using only the digits $0$ and $1$, but  also other bases play important roles in various contexts.

Also non-integer numbers  can be represented as   linear combinations of  powers  of  a  given base $m$.   We   can write 
\begin{equation}\label{e-base m-2}
	x = x_k m^k + \cdots + x_1 m + x_0 + \frac{x_{-1}}{m  } + \frac{x_{-2}}{m^2  }+ \cdots,
\end{equation}
where each  $x_j$ belongs to $\Z_m$. 
We denote with $$
[x]_m := [x_k \cdots x_1 x_0.x_{-1}x_{-2}\cdots]_m,
$$
the  {\it base $m$ representation}  of    $x$ and with  $ [x_i] _m :=x_i$    the $i$-th digit  in this   representation.  When there is no risk of confusion, we will just write $x= [x_k \cdots x_1 x_0.x_{-1}x_{-2}\cdots]_m$ and $x_i=[x_i]_m$.

\medskip
We say that  $x$  has a \emph{finite base $m$ representation},  or:  $x$ is a \emph{finite     decimal in base $m$, } if its base $m$ representation  contains only finitely many nonzero  digits.

It is not difficult to verify that a number cannot have two distinct finite   representations in base $m$; however,  a number can have   a finite and an infinite representation in the same base. For instance, in base ten  we have  $1 = 0.9999\ldots$. In base two, the fraction $\frac{1}{2}$ can be represented either as   $[0.100\ldots]_2$ or as   $[0.01111\ldots]_2$.

\medskip
When multiplying or dividing   decimal numbers by a power of  the base, the position of the decimal point shifts either to the right (when multiplying) or to the left (when dividing). For instance,   if $x=  [0.\,x_1x_2x_3 \cdots ]_m $, then  $m^2 x = [x_1x_2.\,x_3 ...  ]_m$ and $\frac 1{m^2} x = [0.00x_1x_2 x_3  \cdots  ]_m$.

Arithmetic operations such as addition and multiplication between finite     decimals in base $m$  are carried out digit by digit, starting from the rightmost digit. When  the result   exceeds $m - 1$, the excess is carried over to the next digit to the left- a procedure similar to the familiar carrying process in base ten.
For instance, the sum of  $x = [10.2]_3$ and $y = [21.1]_3$  is
$ 
x + y =  [102]_3.
$

\medskip

We have represented  $\mathbb{Z}_m$  by the set of integers  $ \{0, 1, \ldots, m-1\}$; however, since two integers $k$ and $k + \ell m$ (for any $\ell \in \mathbb{Z}$) belong to the same congruence class modulo $m$, the elements of $\mathbb{Z}_m$ can be represented using many different sets of symbols. 

We give the following definition: given  integers   $m>2$   and    $b\in [1, \frac m2]$,  the   \emph{$b$-balanced base $m$ representation} of $\mathbb{Z}_m$  is  the set
$$
\mathbb{Z}_m^{(b)} := \{-b,\, -b+1,\, \ldots,\, m-b-1\}.
$$
This set contains exactly $m$ consecutive integers and, by construction, each element corresponds uniquely to a residue class modulo $m$.  
Different choices of $b$ yield different  representations of $\Z_m$, which became particularly useful    when symmetry around zero is desired.

A remarkable example  of  balanced  base is the $1 $-balanced base three, also  called  {\it balanced ternary  base}.  In this base,   numbers can be written using the symbols  $0$,  $1$, $-1$.  
The balanced ternary representation  offers computational  advantages over the standard  base 3 representation, such as simpler rules for performing  arithmetic operations. 
It is   used in   some analog circuit design   where signals have three states (positive, zero, and negative) and in   modern machine learning algorithms. See e.g. \cite{FT}, \cite{TT2} and the references cited  there.

\medskip
Addition and multiplication in  $\Z_m^{(b)} $ can be   defined in the usual way,  and the   results are taken in $b$-balanced   base  $m$.  
For instance,   $2+1=3$ in $\Z_5 $, but  $2+1=-2$  in  $\Z_5^{(2)}=\{-2,-1, 0, 1, 2\}$. 
With these operations,    $\Z_m^{(b)}$  forms a    commutative rings; when $m$ is prime,  it is also a fields  because every nonzero element has a multiplicative inverse.  

\medskip
As for the standard base $m$,  integers can be identified with polynomials in  $m$ with coefficients in  $\Z_m^{(b)}$ and any   number $x\in\R$ can be written as
in \eqref{e-base m-2}   with  $x_j\in\Z_m^{(b)}.$ 

For a given $x\in\R$, we denote with 
$$
[x]_m^b := [x_k \cdots x_1 x_0.x_{-1}x_{-2}\cdots]_m^b,
$$
the  {\it $b$-balanced base $m$ representation}  of    $x$, and with  $ [x_i] _m^b :=x_i$  the $i$-th  digit in this representation.

As for the  standard base $m$, we  will just write $x=[x_k\cdots  x_{1}  x_0.\, x_{ -1} x_{-2}\cdots ]_m^b$  and $x_i=[x_i]_{m}^b$ when  there is no risk of ambiguity. When $x\in [0,1]$,  we may write $x=[0. x_1x_2\, \cdots ]_m^b$ instead of $ [0. x_{-1}x_{-2}\,\cdots]_m^b$.

\medskip
When $x$ is an integer,    the  $b$-balanced base $m$ representation of $x$ can be obtained as follows: 
let  $ n_0= x $  and   let   $r_0$ be the reminder of the division of $n_0 $ by $m$.  If  $r_0\in [-b, m-b-1] $,  set  $ x _0=r_0$, and if  $r_0 \ge  m-b$ or $r_0<-b$, set  $ x _0=r_0-m$. Then, set  $  n_{1} = (n_0- x_0)/m$,   replace $n_0$ with $n_1$    and     repeat the process  until   $n_k=0$.  We obtain  $x=[x_k \cdots x_1 x_0]_m^b$.

For instance,  the   balanced ternary   representation of $ x=14$  can be produced from the following steps:
\begin{itemize}
	\item 	$  n_0=14 = 3*4+2$, so $r_0=2$, $ x_0= -1,\ n_1= (14-(-1))/3=5$.
	\item $n_1= 5= 3*1+2$, so $r_1=2,\, x _1=-1,\ n_2= (5-(-1))/3=2$.
	\item  $n_2=2=3*0+2$, so $ r_2=2,\, x _2= -1,\ n_3= (2-(-1))/3=1$.
	\item  $n_3=1=3*0+1$, so $ r_3=x _3=  1,\ n_4= 0$.
\end{itemize}
Letting   $-1=\ov 1$,   we  obtain $14=[1\ov 1\, \ov 1 \, \ov 1]_3^1$.   Note  that
$-14=[\ov 1 \,1\,  1 \,  1]_3^1$.

%We will use  $\Z_m^{(b)} $ also to denote the set of symbols $\{-b, \, -b+1,\, \cdots  m-b+1\}$ and also to denote the  ring  of equivalence classes mod.  $m$  assuming  the distinction is always clear from  the context.

%In general  example, when $m$ is odd and $b = \frac{m-1}{2}$, we obtain the commonly used \emph{balanced base $m$} representation, which is symmetric about zero.

%%
%$$[ x]_{m_b}:=[x_k\cdots  x_{1}  x_0.\, x_{ -1} x_{-2}\cdots ]_m^b  $$ 
%the  {\it b-balanced base $m$ representation}  of    $x$ and  we refer to  $x_i := %[x_i]_{m_b}$  as to  the $i$-th digit of $x$ in b-balanced base $m$.  

\medskip

Addition and multiplication between  finite   balanced base $m$ decimals are performed digit by digit, starting from the rightmost digit. When the result exceeds $m-b-1$    or it is  less than $-b$,   the excess  (which can be positive or negative) is carried over to the next digit to the left.
For instance,  
given $x = [2.  \ov  1 \, \ov 2 ]_5^2$ and $ y = [ 1  2. 1 \ov 2]_5^2$,    where $\ov 1$ and $\ov 2$ denote  $-1$  and $ -2$, we have
$x+y = [ 2 \ov 1.\ov 1\, 1]_5^2.$

\medskip

It is natural to ask the following question:

\medskip
{\it When can the sum of two real numbers, expressed in base $ m $ (or in the $ b $-balanced base $ m $), be computed digit by digit without   any carries?}

\medskip

In base $ m $, this question can be reformulated as follows: for which $ x,\, y \ge 0$ does the equality
%\begin{equation}\label{e-no-carry-cond}
$	[(x + y)_i]_m = [x_i]_m + [y_i]_m  $ 
%\end{equation}
hold for every $i\in\Z$? 

Clearly, this  "no-carry"  condition    is only possible when 
\begin{equation}\label{e-no-carry-cond2}
	0 \leq [x_i]_m + [y_i]_m \leq m - 1 \quad \text{for all } i \in \mathbb{Z}.
\end{equation}

In  $ b $-balanced base $ m $ representation,  the equality 
$
[(x + y)_i]_m^b = [x_i]_m^b + [y_i]_m^b
$
holds for all $i\in\Z$ if and only if 
\begin{equation}\label{e-no-carry-cond2b}
	-b \leq [x_i]_m^b + [y_i]_m^b \leq m - 1 - b \quad \text{for all } i \in \mathbb{Z}.
\end{equation}
Answering the question above  leads to the construction of  new families of fractals    that  include the   Sierpiński triangle as a special case.

\section{ The triangle fractals}

In this section  we construct the Sierpiński triangle and characterize it in terms of the base-2 representations of the coordinates of its points. We then define a new family of triangular fractals whose points can be described using the base $m$ representations of their coordinates.

\medskip

We  will  use the following notation: 
given a set $D\subset \R^2$,  a scalar $t>0$ and a vector $v\in \R^2$,  we  denote by  $tD =\{tP,\,  P\in D\}$ and by $v+D=\{v+P, \ P\in D\}$ the dilation and the translation of $D$.

The  Sierpiński triangle can be constructed  as follows. Let     $S_0 $ the   closed triangle with vertices $(0,0), \ (0,1)$ and $( 1,0)$;   for every $j\ge 0$ we    define
$$\mbox{$ S_{j+1}= \frac 12S_j \cup(\frac 12S_j +(\frac 12, 0) ) \cup (\frac 12S_j  +(0,\,\frac 12 ))$} 
$$ 
(see Figure 2).  We can also write:
\begin{equation}\label{e-def-Sj} S_{j+1}= \bigcup_{h, k \in\{0,1\} \atop{ k +h\leq 1}}\hskip -.5 cm \mbox{$\big(\frac 12 S_j+ (\frac k2, \, \frac h2)\big)$}. 
\end{equation} 
%
%It is easy to verify that $S_{j+1}\subset S_j$ and that  each set  $S_j$ consists of $3^j$ right triangles.

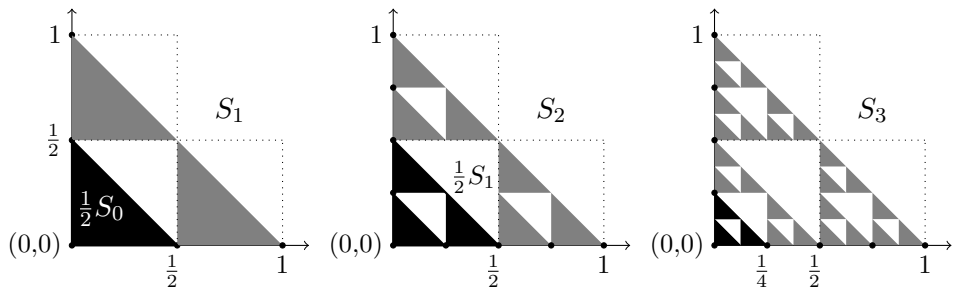
\begin{figure}
	\hskip - .2 cm
	\begin{tikzpicture} [scale=0.7]
		\draw[black, ->] (0,0) -- (4.5,0);
		\draw[black, ->] (0,0) -- (0, 4.5);

		%\draw  (0,3.1)  node[black, left]{\tiny{y=0.1...}};
		\draw[fill=black]  (0,2)  circle(0.05) node[black, left]{ \small{$\frac 12$} };
		
		\draw[fill=black]  (0,0)  circle(0.05) node[black, left]{ \small{(0,0)}};
		\draw[fill=black]  (0,4)  circle(0.05) node[black, left]{ \small{1}};
		\draw[fill=black]  (2,0)  circle(0.05);
		\draw (1.9, 0) node[black, below ]{ \small{$\frac 12$} }   ;
		
		%\draw[fill=black]   circle (0.05) node[below]; 
		\fill[gray] (2,0)--(4,0)--(2,2)--cycle; 
		%\draw (0, 1 ) node[white, right]{\tiny{y= .0...}};
		%\draw (0,3) node[white, right]{\tiny{y= .1...}};
		\fill[black] (0,0)--(2,0)--(0,2)--cycle;
		
		\fill[gray] (0,4) -- (2,2) -- (0,2) -- cycle; 
		
		\draw[black, dotted ] (2,0) -- (2,4);
		\draw[black, dotted] (4,0) -- (4, 2 );
		\draw[black, dotted] (0,2) -- (4,2);
		\draw[black, dotted] (0,4) -- (2, 4 );
		\draw[fill=black]  (4,0)  circle(0.05) node[black, below]{ \small{1} };
		\draw (-.1, .7) node[white, right] { {\bf $\frac 12S_0 $}};
		\draw[black, thick] (0,0)--(2,0)--(0,2)--cycle;
		\draw (3,  3) node[black, below] { {\bf $S_1$}};
	\end{tikzpicture}
	\begin{tikzpicture}[scale=0.7]
		\draw[black, ->] (0,0) -- (4.5,0);
		\draw[black, ->] (0,0) -- (0, 4.5);
		\fill[gray] (0,4) -- (1,3) -- (0,3) -- cycle;
		\fill[gray] (0,3) -- (1,2) -- (0,2) -- cycle;
		
		\fill[gray] (1,2 ) -- (2,2) -- (1,3) -- cycle;
		
		\fill[gray] (2,2) -- (3,1) -- (2,1) -- cycle;
		\fill[gray] (2,1) -- (3,0) -- (2,0) -- cycle;
		\fill[gray] (3,0 ) -- (4,0) -- (3,1) -- cycle;
		\fill[black] (0,0)--(2,0)--(0,2)--cycle; 
		\fill[white] (1,0)--(1,1)--(0,1)--cycle;  
		
		\draw[black, dotted ] (2,0) -- (2,4);
		\draw[black, dotted] (4,0) -- (4, 2);
		\draw[black, dotted] (0,2) -- (4,2);
		\draw[black, dotted] (0,4) -- (2, 4 );
		
		\draw[fill=black]  (0,0)  circle(0.05) node[black, left]{ \small{(0,0)}};
		\draw[fill=black]  (0,1)  circle(0.05) %node[black, left]{ \tiny{(0,0)}}
		;
		\draw[fill=black]  (0,2)  circle(0.05); 
		\draw[fill=black]  (0,3)  circle(0.05);  
		\draw[fill=black]  (1,0)  circle(0.05) %node[black, left]{ \tiny{(0,0)}}
		;
		\draw[fill=black]  (0,4)  circle(0.05) node[black,  left]{ {\small 1}};
		\draw[fill=black]  (2,0)  circle(0.05); 
		\draw[fill=black]  (3,0)  circle(0.05);  
		\draw[fill=black]  (4,0)  circle(0.05)node[black,  below]{ {\small 1}};  
		\draw[black, thick] (0,0)--(2,0)--(0,2)--cycle;
		%\draw (-.1 ,  .3) node[white, right] { \tiny{\bf $T_2$}};
		\draw (3,  3) node[black, below] { {\bf $S_2$}};
		\draw (1.9, 0) node[black, below ]{ \small{$\frac 12$} }   ;
		\draw (.9,  1.3) node[black, right] {\small{\bf $\frac 12S_1 $}};
	\end{tikzpicture}
	\begin{tikzpicture}[scale=0.7]
		\draw[black, ->] (0,0) -- (4.5,0);
		\draw[black, ->] (0,0) -- (0, 4.5);
		
		\newcommand{\smallfigure}{\fill[black] (2,0)--(4,0)--(2,2)--cycle; 
			%\draw (0, 1 ) node[white, right]{\tiny{y= .0...}};
			%\draw (0,3) node[white, right]{\tiny{y= .1...}};
			\fill  (0,0)--(2,0)--(0,2)--cycle;
			
			\fill  (0,4) -- (2,2) -- (0,2) -- cycle; }
		\newcommand{\Myfigure} {
			\fill[gray] (0,4) -- (1,3) -- (0,3) -- cycle;
			\fill[gray] (0,3) -- (1,2) -- (0,2) -- cycle;
			
			\fill[gray] (1,2 ) -- (2,2) -- (1,3) -- cycle;
			
			\fill[gray] (2,2) -- (3,1) -- (2,1) -- cycle;
			\fill[gray] (2,1) -- (3,0) -- (2,0) -- cycle;
			\fill[gray] (3,0 ) -- (4,0) -- (3,1) -- cycle;
			\fill[gray] (0,0)--(2,0)--(0,2)--cycle; 
			\fill[white] (1,0)--(1,1)--(0,1)--cycle; } 

		\foreach \x/\y in {0/0, 2/0,   0/2} {
			\begin{scope}[shift={(\x,\y)}, scale=1/2]
				\Myfigure
		\end{scope}}
		
		\foreach \x/\y in {0/0, 2/0,   0/2}{
			\begin{scope}[scale=1/4, fill=black]
				black]
				\smallfigure
		\end{scope} }
		
		\draw[black, dotted ] (2,0) -- (2,4);
		\draw[black, dotted] (4,0) -- (4, 2);
		\draw[black, dotted] (0,2) -- (4,2);
		\draw[black, dotted] (0,4) -- (2, 4 );
		
		\draw[fill=black]  (0,0)  circle(0.05) node[black, left]{ \small{(0,0)}};
		\draw[fill=black]  (0,1)  circle(0.05) %node[black, left]{ \tiny{(0,0)}}
		;
		\draw[fill=black]  (0,2)  circle(0.05); 
		\draw[fill=black]  (0,3)  circle(0.05);  
		\draw[fill=black]  (1,0)  circle(0.05) %node[black, left]{ \tiny{(0,0)}}
		;
		\draw[fill=black]  (0,4)  circle(0.05) node[black,  left]{ {\small 1}};
		\draw[fill=black]  (2,0)  circle(0.05); 
		\draw[fill=black]  (3,0)  circle(0.05);  
		\draw[fill=black]  (4,0)  circle(0.05)node[black,  below]{ {\small 1}};  
		%\draw[black, thick] (0,0)--(2,0)--(0,2)--cycle;
		%\draw (-.1 ,  .3) node[white, right] { \tiny{\bf $T_2$}};
		\draw (3,  3) node[black, below] { {\bf $S_3$}};
		\draw (1.9, 0) node[black, below ]{ \small{$\frac 12$} }   ;
		\draw ( .9, 0) node[black, below ]{ \small{$\frac 14$} }   ;
		%\draw (.9,  1.3) node[black, right] {\small{\bf $\frac 12S_1 $}};
	\end{tikzpicture}
	\caption{ \small{   The first steps of the construction of the Sierpiński triangle }}
\end{figure}

The Sierpiński triangle  is   defined as 
$  \dsize S_T := \bigcap_{j =1}^{\infty} S_j.$  

Note that $S_T$ is closed because it is intersection of closed sets. 
For every  $  i\ge 0$, 
we can  denote by $ \tau_i(P)  $   the   triangle  in the set $S_i$ that contains $P$  and   with   $v_i(P)$  the lower left vertex of $ \tau_i(P)$.    Since $P=\bigcap_{j=1}^{\infty} \tau_j(P)$, we   have  that $\dsize P=\lim_{j\to\infty}v_j(P)$.  
The sequence $\{v_j(P)\}_{j\in\N}$ can be used to identify the position of $P$ in  $S_T$; for details,  see e.g.  the construction of the Sierpiński triangle in \cite{S}.

%\ld{ Removed the description on how to find points in the triangle}

\medskip
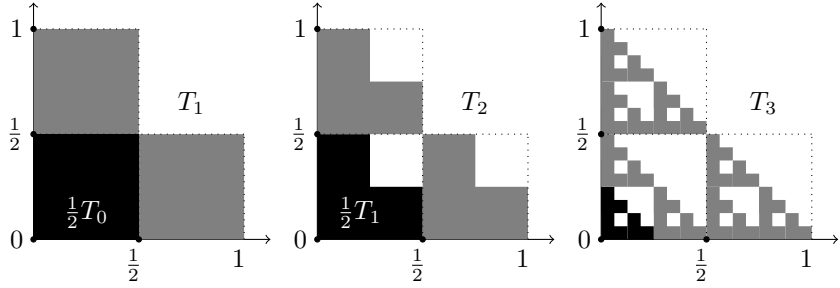
\begin{figure}
	
	\begin{tikzpicture}[scale=0.7]
		\draw[black, ->] (0,0) -- (4.5,0);
		\draw[black, ->] (0,0) -- (0, 4.5);

		\fill[gray] (0,0)-- (4,0) -- (4,2) -- (2,2) -- (2,4)-- (0,4)--cycle;
		\fill[black](0,0)-- (2,0) -- (2,2) -- (0,2)--cycle;
		
		\draw[black, dotted ] (2,0) -- (2,4);
		\draw[black, dotted] (4,0) -- (4, 2 );
		\draw[black, dotted] (0,2) -- (4,2);
		\draw[black, dotted] (0,4) -- (2, 4 );
		
		\draw (1,1) node[white, below ]{ \small{$\frac 12 \Q_0 $} } ; 
		\draw[fill=black]  (0,2)  circle(0.05) node[black, left]{ \small{$\frac 12$} };
		
		\draw[fill=black]  (0,0)  circle(0.05) node[black, left]{ \small{0}};
		\draw[fill=black]  (0,4)  circle(0.05) node[black, left]{ \small{1}};
		\draw[fill=black]  (2,0)  circle(0.05);
		\draw (1.9, 0) node[black, below ]{ \small{$\frac 12$} } ; 
		\draw (3.9, 0) node[black, below ]{ \small{$1$} };
		\draw (3,3) node[black, below]{ \small{$\Q_1$} };
	\end{tikzpicture}
	\begin{tikzpicture}[scale=0.7]
		\draw[black, ->] (0,0) -- (4.5,0);
		\draw[black, ->] (0,0) -- (0, 4.5);

		\fill[black] (0,0)-- (2,0) -- (2,1) -- (1,1) -- (1,2)-- (0,2)--cycle;
		\fill[gray] (2,0)-- (4,0) -- (4,1) -- (3,1) -- (3,2)-- (2,2)--cycle;
		\fill[gray] (0,2)-- (2,2) -- (2,3) -- (1,3) -- (1,4)-- (0,4)--cycle;
		
		\draw[black, dotted ] (2,0) -- (2,4);
		\draw[black, dotted] (4,0) -- (4, 2 );
		\draw[black, dotted] (0,2) -- (4,2);
		\draw[black, dotted] (0,4) -- (2, 4 );
		
		\draw (.8,1) node[white, below ]{ \small{$\frac 12\Q_1 $} } ; 
		\draw[fill=black]  (0,2)  circle(0.05) node[black, left]{ \small{$\frac 12$} };
		
		\draw[fill=black]  (0,0)  circle(0.05) node[black, left]{ \small{0}};
		\draw[fill=black]  (0,4)  circle(0.05) node[black, left]{ \small{1}};
		\draw[fill=black]  (2,0)  circle(0.05);
		\draw (1.9, 0) node[black, below ]{ \small{$\frac 12$} } ; 
		\draw (3.9, 0) node[black, below ]{ \small{$1$} };
		\draw (3,3) node[black, below]{ \small{$\Q_2$} };
	\end{tikzpicture}
	\begin{tikzpicture}[scale=0.7]
		\draw[black, ->] (0,0) -- (4.5,0);
		\draw[black, ->] (0,0) -- (0, 4.5);

		\newcommand{\myshape}{
			\fill[gray] (0,0)-- (1/2,0) -- (1/2,1/4) -- (1/4,1/4) -- (1/4,1/2)-- (0,1)--cycle;
			\fill[gray] (1/2,0)-- (1,0) -- (1,1/4) -- (3/4,1/4) -- (3/4,1/2)-- (1/2,1/2)--cycle;
			\fill[gray] (0,1/2)-- (1/2,1/2) -- (1/2,3/4) -- (1/4,3/4) -- (1/4,1)-- (0,1)--cycle;
		}
		
		% Original shape at (0,0)
		\begin{scope}
			\fill[black] (0,0)-- (1/2,0) -- (1/2,1/4) -- (1/4,1/4) -- (1/4,1/2)-- (0,1)--cycle;
			\fill[black] (1/2,0)-- (1,0) -- (1,1/4) -- (3/4,1/4) -- (3/4,1/2)-- (1/2,1/2)--cycle;
			\fill[black] (0,1/2)-- (1/2,1/2) -- (1/2,3/4) -- (1/4,3/4) -- (1/4,1)-- (0,1)--cycle;
		\end{scope}
		
		% Translated copy by (0,2)
		\begin{scope}[shift={(0,1)}]
			\myshape
		\end{scope}

		\begin{scope}[shift={(1,0)}]
			\myshape
		\end{scope}
		
		\begin{scope}[shift={(2,0)}]
			\myshape
		\end{scope}
		\begin{scope}[shift={(3,0)} ]
			\myshape
		\end{scope}
		\begin{scope}[shift={(2,1)} ]
			\myshape
		\end{scope}
		\begin{scope}[shift={(0,2)} ]
			\myshape
		\end{scope}
		\begin{scope}[shift={(0,3 )} ]
			\myshape
		\end{scope}
		\begin{scope}[shift={(1,2)} ]
			\myshape
		\end{scope}
		\draw[black, dotted ] (2,0) -- (2,4);
		\draw[black, dotted] (4,0) -- (4, 2 );
		\draw[black, dotted] (0,2) -- (4,2);
		\draw[black, dotted] (0,4) -- (2, 4 );
		%\draw (.8,1) node[white, below ]{ \small{$Q_1'$} } ; 
		\draw[fill=black]  (0,2)  circle(0.05) node[black, left]{ \small{$\frac 12$} };
		
		\draw[fill=black]  (0,0)  circle(0.05) node[black, left]{ \small{0  }};
		\draw[fill=black]  (0,4)  circle(0.05) node[black, left]{ \small{1}};
		\draw[fill=black]  (2,0)  circle(0.05);
		\draw (1.9, 0) node[black, below ]{ \small{$\frac 12$} } ; 
		\draw (3.9, 0) node[black, below ]{ \small{$1$} };
		\draw (3,3) node[black, below]{ \small{ $\Q_3$} };
		
	\end{tikzpicture}
	
	\caption{ \small{    Construction of  the Sierpiński triangle using squares}}
\end{figure}
The  Sierpiński triangle can also be constructed as the limit of a decreasing sequence of union of squares.    We let $T(0)=Q=[0,1]\times[0,1]$,  and
\begin{equation}\label{e-def-Tj} 
	T(j+1) = \bigcup_{h, k \in\{0,1\} \atop{  k +h\leq 1}}\hskip -.5 cm \mbox{$(\frac 12  T(j)+ (\frac k2, \, \frac h2))$}, \quad  j\ge 0 .
\end{equation}
%(compare this with \eqref{e-def-Sj}).
%We show  that $S_T =  \bigcap_{j =0}^{\infty} \Q_j$.  
%
See Figure 3.  Let us show that $S_T= \bigcap_{j =1}^{\infty} T(j)$;  recall that each point $P\in S_T$ lies at the intersection of a decreasing sequence of triangles  $ \{ \tau_i(P) \}_{i \in \mathbb{N}} $,  with  $ \tau_i(P) \subset S_i $, where the $S_i$ are defined in  \eqref{e-def-Sj}.
Our construction shows that for  each triangle $ \tau_i(P) $    there exist   squares $q_j(P) \subset T(j)$ and $ q_{j+1}(P)\subset T(j+1) $, both containing $P$,  for which 
$$
q_{i+1 }(P) \subset \tau_i(P) \subset q_{i }(P).  
$$
(see Figures 2 and 3).
Thus, $P=  \bigcap_{j =0}^{\infty} q_j(P)$ and $S_T=\bigcap_{j =0}^{\infty} T(j)$, as  required.

\medskip

In order to  characterize  the points in the Sierpiński triangle in terms of their base $2$ representations,
we  need the following  

\begin{Lemma}\label{L-1} 
	For every $n \geq 0$ and $m \geq 2$,  the  points $(x, y) \in \frac{1}{m^n}Q$ have base $m$ representation 
	\begin{equation}\label{e-base 1}
		x = [x_0. x_1 x_2  \cdots ]_m, \quad y = [x_0. y_1 y_2 \cdots ]_m,\ \mbox{with $x_i = y_i = 0$  when  $ i \leq n$}.
	\end{equation}
	Conversely, every  point $(x, y) $ for which \eqref{e-base 1} holds belongs to $\frac{1}{m^n}Q$
\end{Lemma}

\begin{proof}
	Since multiplying  by $m^{-n}$ shifts the decimal point $n$ places to the right in base $m$ representation,  it suffices to prove the lemma for $n=0$. 
	
	Any number  $x = [0.x_1 x_2  \cdots]_m$  lies in the interval $[0, 1]$ because 
	\begin{align*}0\leq 
		x &= \frac{x_1}{m } + \frac{x_2}{m^2} + \cdots  \leq \frac{m-1}{m } + \frac{m-1}{m^2} + \cdots \\  &=\frac{m-1}{m } \cdot \frac{1}{1 - \frac{1}{m}} = 1.
	\end{align*}
	Conversely,    %all points  $x\in [0,1]$ have a  base $m$ representation   for which  $[x_0]_m=0$; indeed,  
	the base $m$ representation   of any   $x\in  [0, \, 1)$  is  necessarily in the form   $x = [0.  x_1 x_1 \cdots]_m $, and   $x=1$  has  also the infinite decimal representation $1= [0. \, (m-1)\,(m-1)\cdots ] _m$. 
	We can conclude   that  $(x, y) \in Q $  if and only if  $x$ and $y$ have  a base $m$ representation   of the form 
	$  
	x = [0.  x_1 x_1 \cdots]_m, \quad y = [0.y_1 y_2 \cdots]_m $,  and the Lemma is proved.
\end{proof}

\begin{Thm}\label{T-Main}
	% Let $x=[0.x_1x_2\cdots]_2$ and $y=[0.y_1y_2\cdots]_2$.
	The point   $(x,y) \in Q$   lies in the  Sierpiński triangle $S_T  $ if and only if   the  base $2$ representations  of $x$ and $y$ satisfies 
	\begin{equation}\label{e-cond-Tn}
		0\leq [x_i]_2 + [y_i]_2 \leq 1  
	\end{equation}  for every $i\in\N$. 
\end{Thm}

%The characterization  of the points of  $S_T$ in terms of   their base $2$ representation  is well known, but we give a  full proof of  Theorem \ref{T-Main}  because it will serve  as a model  for analogous results in base $m$.

\begin{proof}
	We have shown that the Sierpiński triangle can be written as
	$
	S_T = \bigcap_{n=1}^\infty   T(n)
	$
	where the sets $ T(n)$ are defined  in \eqref{e-def-Tj}.
	
	We prove by induction on $n$ that every point $(x, y) \in  T(n)$ satisfies
	\eqref{e-cond-Tn} 
	for all $1\leq i\leq n$, starting with  the base case $n=1$.
	
	\medskip
	By Lemma~\ref{L-1}, every point $(x, y) \in \frac{1}{2} Q = [0, \tfrac{1}{2}] \times [0, \tfrac{1}{2}]$ has base $2$ representation of the form
	$
	x = [0.0 x_2 x_3 \cdots]_2, \quad y = [0.0 y_2 y_3 \cdots]_2.
	$
	Since adding $\tfrac{1}{2}$  only affects  the first decimal digit of the base 2  representation,  the set $
	T(1)= \tfrac{1}{2} Q \cup \left( \tfrac{1}{2} Q + (0, \tfrac{1}{2}) \right) \cup \left( \tfrac{1}{2} Q + (\tfrac{1}{2}, 0) \right)
	$
	consists of    points $(x, y)$  for which $0 \leq [x_1]_2 + [y_1]_2 \leq 1$. In particular, $ T(1)$ contains all points  $(x,y)$  with 
	$
	x = [0.x_1 0 0 \cdots]_2, \quad y = [0.y_1 0 0 \cdots]_2 $, and $0\leq  [x_1]_2 + [y_1]_2 \leq 1.
	$
	
	\medskip  
	
	We now assume that the points $ (x, y) \in T(n) $ satisfy \eqref{e-cond-Tn} for all $ 1 \leq i \leq n $. 
	We also assume that $ T(n) $ contains all points $ (x, y) \in Q $ that satisfy the additional condition  
	$
	[x_k]_2 = [y_k]_2 = 0 \quad \text{for all } k \geq n+1.
	$
	
	Scaling $ T(n) $ by $ \tfrac{1}{2} $ shifts the decimal point in the base 2 representation. 
	Therefore, the points $ (x, y) \in \tfrac{1}{2} T(n) $ are of the form
	\[
	x = [0.0 x_2 x_3 \cdots]_2, \quad y = [0.0 y_2 y_3 \cdots]_2,
	\]
	with the condition $ 0 \leq x_j + y_j \leq 1 $ for all $ 2 \leq j \leq n+1 $. 
	
	Furthermore, $ \tfrac{1}{2} T(n) $ contains all points $ (x, y) $ that satisfy the additional condition 
	$
	[x_k]_2 = [y_k]_2 = 0 \quad \text{for all } k \geq n+2.
	$
	
	Thus,
	$
	T(n+1) = \tfrac{1}{2} T(n) \cup \left( \tfrac{1}{2} T(n) + (0, \tfrac{1}{2}) \right) 
	\cup \left( \tfrac{1}{2} T(n) + (\tfrac{1}{2}, 0) \right)
	$
	contains all points $ (x, y) $ that satisfy \eqref{e-cond-Tn} for $ 1 \leq i \leq n+1 $, 
	as well as all points that satisfy   the additional condition  $ [x_k]_2 = [y_k]_2 = 0 $ for all $ k \geq n+2 $

	\medskip	
	The Sierpiński  triangle $S_T$ is the intersection of  the   $T(n)$,    and so  every point $(x, y) \in S_T$ satisfies \eqref{e-cond-Tn} for all $i \in \mathbb{N}$.
	
	Conversely, assume that   $P = (a, b)$ satisfies \eqref{e-cond-Tn} for all $i \in \mathbb{N}$; for every $n\in \N$ we can let   $P_n=(X_n,\,Y_n)$  with 
	$ X_n= [0.a_1 a_2 \cdots a_n \, 0\, 0 \cdots]_2 $, $Y_n= \ [0.b_1 b_2 \cdots b_n\, 0\, 0 \cdots]_2 $.
	
	We have proved that $P_n \in T(k)$ for all $k \geq n$,  and so  $P_n \in S_T$.  Since 
	$\dsize\lim_{n\to\infty} P_n=P$  and $S_T $ is closed, it follows that $P \in S_T$.
\end{proof}

\medskip
The construction of the Sierpiński triangle and the proof of Theorem~\ref{T-Main} can be extended to a broader class of planar fractals.
% which can be characterized in terms of  the  base $m$ representation of their   points.

Let  $m\ge 2$ be a given base. We   let $T_m(0)=Q$  and  for every $j\ge 0$, we  define
\begin{equation}\label{e-Qn}  T_m (j+1)=\bigcup_{h, k \ge 0 \atop{0\leq k +h\leq m-1}}\hskip -.5 cm \mbox{$(\frac 1m  T_m(j)+ (\frac km, \, \frac hm))$}. 
\end{equation} 
The set $T_m(1)$ contains 
$ s_m:= m+ (m-1)+\cdots +1 = \frac{m(m+1)}{2}$squares of side $\frac 1m, 
$ (see Figure 4) and  we can verify  by induction that $T _m(n)$ contains $(s_m)^n $ squares of side $\frac {1}{m^n}$.  
Our construction shows that the sequence $\{T_m(n)\}_{n\in\N}$ is decreasing.

We define   the {\it $m$-th   fractal triangle } as   the set  $$\mathcal{T}_m := \bigcap_{n=1}^\infty T_m(n).
$$
In this notation, the classical Sierpiński triangle is   ${\mathcal T}_2$.

\bigskip
\begin{figure}
	\begin{tikzpicture}[scale=0.5]
		\draw[black, ->] (0,0) -- (6.5,0);
		\draw[black, ->] (0,0) -- (0, 6.5);
		\fill[gray] (0,0)-- (6,0) -- (6,2) -- (4,2) -- (4,4)-- (2,4)--(2,6)--(0,6)--cycle;
		
		%%%%%%%%
		\draw[fill=black]  (0,0)  circle(0.05) node[black, left]{ \small{0}};
		\draw[fill=black]  (0,2)  circle(0.05) node[black, left]{ \small{$\frac 13$}};
		\draw[fill=black]  (0,4)  circle(0.05) node[black, left]{ \small{$\frac 23$}};
		\draw[fill=black]  (0,6)  circle(0.05) node[black, left]{ \small{$  1$}};
		\draw[fill=black]  (2,0)  circle(0.05);
		\draw (1.9, 0) node[black, below ]{ \small{$\frac 13$} } ; 
		\draw (3.9, 0) node[black, below ]{ \small{$\frac 23$} };
		\draw (5.9, 0) node[black, below ]{ \small{$1$} };

		\draw[black, thick, dotted] (0,2) -- (6,2);
		\draw[black,thick,dotted] (0,4) -- (4,4);
		\draw[black,thick,dotted] (2,0) -- (2,6);
		\draw[black,thick,dotted] (4,0) -- (4,4);
		\draw (4.9 ,4.9) node[black, below]{ \small{$T_3(1)$} };
		\draw[black, fill] (0,0)--(2,0)--(2,2)--(0, 2)--cycle;
		\draw ( 1,1.5) node[white, below]{ \small{$\frac 13 Q$} };
		%%%%%%%%%
	\end{tikzpicture}
	\begin{tikzpicture}[scale=0.5]
		\draw[black, ->] (0,0) -- (6.5,0);
		\draw[black, ->] (0,0) -- (0, 6.5);

		\newcommand{\myshape}{
			\fill[gray] (0,0)-- (6,0) -- (6,2) -- (4,2) -- (4,4)-- (2,4)--(2,6)--(0,6)--cycle;
			
		}
		\newcommand{\Myshape }   {
			\foreach \x/\y in {0/0, 2/0, 4/0, 0/2, 2/2, 0/4} {
				\begin{scope}[shift={(\x,\y)}, scale=1/3]
					\myshape
				\end{scope}
		}}
		\begin {scope} 
		\Myshape 
	\end{scope}	
	\fill[black] (0,0)-- (2 ,0) -- (2 ,2/3) -- (4/3,2/3) -- (4/3, 4/3)-- (2/3,4/3)-- (2/3,2)--(0,2)--cycle;
	%%%%%%%%
	\draw[fill=black]  (0,0)  circle(0.05) node[black, left]{ \small{0}};
	\draw[fill=black]  (0,2)  circle(0.05) node[black, left]{ \small{$\frac 13$}};
	\draw[fill=black]  (0,4)  circle(0.05) node[black, left]{ \small{$\frac 23$}};
	\draw[fill=black]  (0,6)  circle(0.05) node[black, left]{ \small{$  1$}};
	\draw[fill=black]  (2,0)  circle(0.05);
	\draw (1.9, 0) node[black, below ]{ \small{$\frac 13$} } ; 
	\draw (3.9, 0) node[black, below ]{ \small{$\frac 23$} };
	\draw (5.9, 0) node[black, below ]{ \small{$1$} };

	\draw[black, thick, dotted] (0,2) -- (6,2);
	\draw[black,thick,dotted] (0,4) -- (4,4);
	\draw[black,thick,dotted] (2,0) -- (2,6);
	\draw[black,thick,dotted] (4,0) -- (4,4);
	\draw[black,thick,dotted] (6,0) -- (6,2);
	\draw[black,thick,dotted] (0,6) -- (2,6);
	\draw (4.9 ,4.9) node[black, below]{ \small{$T_3(2)$} };
	\draw ( .85,1 ) node[white, thick,below]{ \tiny{$\frac 13T_3(1)$} };
	%%%%%%%%%
	
\end{tikzpicture}
\begin{tikzpicture}[scale=0.5]
	\draw[black, ->] (0,0) -- (6.5,0);
	\draw[black, ->] (0,0) -- (0, 6.5);

	\newcommand{\myshape}{
		\fill  (0,0)-- (6,0) -- (6,2) -- (4,2) -- (4,4)-- (2,4)--(2,6)--(0,6)--cycle;
		
	}
	\newcommand{\Myshape }   {
		\foreach \x/\y in {0/0, 2/0, 4/0, 0/2, 2/2, 0/4} {
			\begin{scope}[shift={(\x,\y)}, scale=1/3, fill=gray]
				\myshape
			\end{scope}
	}}
	
	\newcommand{\Mynewshape }   {
		\foreach \x/\y in {0/0, 2/0, 4/0, 0/2, 2/2, 0/4} {
			\begin{scope}[shift={(\x,\y)}, scale=1/3, fill=gray]
				\Myshape
			\end{scope}
	}}
	\begin {scope} [fill=black]
	\Mynewshape 
\end{scope}	
%\begin {scope}[scale=1/3, fill=black]
%\Myshape 
%\end{scope}

%%%%%%%%
\draw[fill=black]  (0,0)  circle(0.05) node[black, left]{ \small{0}};
\draw[fill=black]  (0,2)  circle(0.05) node[black, left]{ \small{$\frac 13$}};
\draw[fill=black]  (0,4)  circle(0.05) node[black, left]{ \small{$\frac 23$}};
\draw[fill=black]  (0,6)  circle(0.05) node[black, left]{ \small{$  1$}};
\draw[fill=black]  (2,0)  circle(0.05);
\draw (1.9, 0) node[black, below ]{ \small{$\frac 13$} } ; 
\draw (3.9, 0) node[black, below ]{ \small{$\frac 23$} };
\draw (5.9, 0) node[black, below ]{ \small{$1$} };

\draw[black, thick, dotted] (0,2) -- (6,2);
\draw[black,thick,dotted] (0,4) -- (4,4);
\draw[black,thick,dotted] (2,0) -- (2,6);
\draw[black,thick,dotted] (4,0) -- (4,4);
\draw[black,thick,dotted] (6,0) -- (6,2);
\draw[black,thick,dotted] (0,6) -- (2,6);
\draw (5 ,4.9) node[black, below]{ \small{$T_3(3)$} };
%

%\draw ( 1,1.5) node[black, below]{ \small{$Q_3$} };
%%%%%%%%%

\end{tikzpicture}
\caption{  The first steps of the construction of  the fractal triangle  ${\mathcal T}^{(3)}$  }
\end{figure}
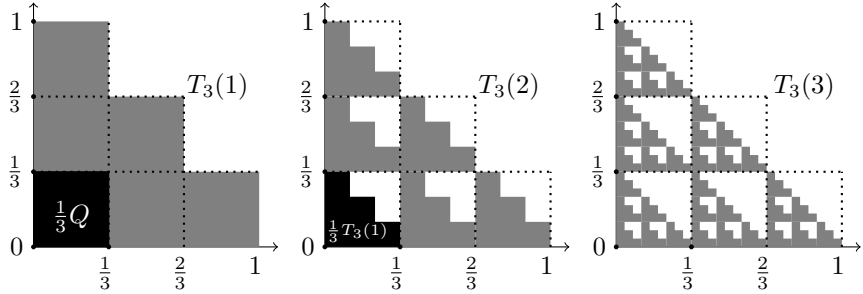

The following theorem allows to  characterize  the points of ${\mathcal T}_m$  in terms of the  base $m$ representation  of their coordinates. 

\begin{Thm}\label{T-Main-m}
%Let $x=[0.x_1x_2\cdots]_m$ and $y=[0.y_1y_2\cdots]_m$.
The point    $(x,y)\in Q $   lies in the   triangle fractal  ${\mathcal T}_m $ if and only if   
the  base $m$ representations  of $x$ and $y$ satisfies 
\begin{equation}\label{e-cond-m}
0\leq [x_i]_m+[y_i ]_m \leq m-1\end{equation} 
for every $i\in\N$.
\end{Thm}

\begin{proof}
%	Letting Let $x=[0.x_1x_2\cdots]_m$ and $y=[0.y_1y_2\cdots]_m$, the identity \eqref{e-cond-m} is equivalent to   $0\leq x_j+y_j\leq m-1$ for every $i\in\N$.
%

By Lemma~\ref{L-1}, the points $ (x, y) $ in the square $ \tfrac{1}{m} Q $ have base $ m $ representations of the form
\[
x = [0.0 x_2 x_3 \cdots ]_m, \qquad y = [0.0 y_2 y_3 \cdots ]_m.
\]
Adding $ \tfrac{k}{m} $ to $ x $ or $ y $ only affects the first decimal digit   in their base $ m $ expansions.
In view of the definition~\eqref{e-Qn}, we conclude that $ T_m(1) $ consists of those points $ (x, y) $ for which
$
0 \leq [x_1]_m + [y_1]_m \leq m - 1.
$
Moreover, $ T_m(1) $ contains all points  $(x,y)$ with 
$
x = [0.x_1 00\cdots]_m, \qquad y = [0.y_1 00\cdots]_m,
$ and $
0 \leq [x_1]_m + [y_1]_m \leq m - 1.
$

The argument used in the proof of Theorem~\ref{T-Main} can be extended to show that $ T_m(j) $ consists of points $ (x, y) $ satisfying~\eqref{e-cond-m} for every $ i \leq j $, as well as all the points that satisfy the additional condition 
$
[x_k]_m = [y_k]_m = 0 \quad \text{for all } k \geq j + 1.
$

We can therefore conclude that all points $ (x, y) \in \mathcal{T}_m $ have base $ m $ representations satisfying~\eqref{e-cond-m} for every $ i \in \mathbb{N} $. Furthermore, the argument from the second part of the proof of Theorem~\ref{T-Main} shows that every point $ (x, y) \in Q $ satisfying~\eqref{e-cond-m} for all $ i \in \mathbb{N} $ belongs to $ \mathcal{T}_m $.

\end{proof}

\section{The  hexagon  fractals  } 

In this section  we introduce a new family of   fractals     and we characterize them in terms of  the $b$-balanced base $m$  representation of the coordinates of  their points.   

Given integers $ m > 2 $ and $ b \in \left[1, \frac{m}{2} \right] $, we define
\begin{equation}\label{e-Lmb}
L_{m}^b := \left\{ (k, h) \in \mathbb{Z} \times \mathbb{Z} \ : \ h,\, k,\, h+k \in [-b,\, m - 1 - b] \right\}.
\end{equation}
(See Figure 5).  Let $ H_m^b(0) = Q $, and 
\begin{equation}\label{e-Hn}
H_{m}^b(j+1) = \bigcup_{(h, k) \in L_{m}^b} \left( \tfrac{1}{m} H_m^b(j) + \left( \tfrac{k}{m}, \tfrac{h}{m} \right) \right), \quad   j \geq 0.
\end{equation}

The \emph{$b$-balanced $m$-hexagon fractal} is the set
\[
\mathcal{H}_m^b := \bigcap_{j=1}^\infty H_{m}^b(j).
\]

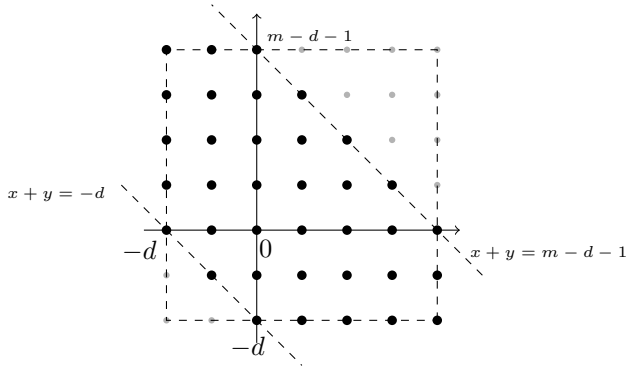
\begin{figure}
\begin{tikzpicture}[scale=0.6]
% Define bounds
\def\xmin{-2}
\def\xmax{4}
\def\ymin{-2}
\def\ymax{4}

% Draw all grid points and highlight those with -2 ≤ x+y ≤ 4
\foreach \x in {\xmin,...,\xmax} {
	\foreach \y in {\ymin,...,\ymax} {
		\pgfmathtruncatemacro{\s}{\x+\y}
		\ifnum \s<-2
		\fill[gray!60] (\x,\y) circle(2pt); % out of range (below)
		\else
		\ifnum \s>4
		\fill[gray!60] (\x,\y) circle(2pt); % out of range (above)
		\else
		\fill[black] (\x,\y) circle(3pt);     % in desired region
		\fi
		\fi
	}
}

% Axes
\draw[->] (\xmin-0.5,0) -- (\xmax+0.5,0)  ;
\draw[->] (0,\ymin-0.5) -- (0,\ymax+0.8)  ;
\draw[black, dashed] (4,-2)--( 4,4)--(-2, 4)--(-2,-2)--cycle; 

%\draw (3, 0) node[black, below ]{ \small{$m-d-1$} } ; 
\draw (.2, 0) node[black, below]{ \small{$0$} };
\draw (-2.7,0) node[black, below]{ \small{ $-d$} };

\draw (0,4.3 ) node[black, right ]{ \tiny{$m-d-1$} } ; 
%\draw (0, 0) node[black, right ]{ \small{$0$} };
\draw (-1,-2.6) node[black, right]{ \small{ $-d$} };
\draw[black, dashed] (5,-1)--(-1,5);
\draw (4.5,-.5 ) node[black, right ]{ \tiny{$x+y=m-d-1$} } ; 
\draw[black, dashed] (-3 , 1)--( 1,-3);
\draw (-3,   .8) node[black, left ]{ \tiny{$x+y= -d$ } } ; 

\end{tikzpicture}
\caption{Points $(x,y)\in [-d, m-d]\times[-d, m -d]$  for which  $-d\leq x+y\leq m-d$}
\end{figure}

The   points in $ L_m^b $ lie within a hexagonal region obtained from the square $ [-b,\, m - b - 1] \times [-b,\, m - b - 1] $ by removing the triangular regions above the line $ x + y = m - b - 1 $ and below the line $ x + y = -b $. Consequently, the sets $ \mathcal{H}_m^b $ exhibit a hexagonal shape, as illustrated in Figures 5 and 6.

It is clear from \eqref{e-Hn} that the sequence $ \{H_m^b(j)\}_{j \geq 0} $ is decreasing;
for every $ j \in \mathbb{N} $, the set $ H_m^b(j) $ consists of $ (\ell_m^b)^j $ squares of side length $ \tfrac{1}{m^j} $, where  $ \ell_m^b $ denotes the cardinality of $ L_m^b $.  

Approximations of the fractal $ \mathcal{H}_3^1 $   have been used to describe certain self-similar graphs that are recursively generated using Kronecker products of the matrix
\[
\begin{pmatrix}
1 & 1 & 0 \\
1 & 1 & 1 \\
0 & 1 & 1
\end{pmatrix}
\]
\cite{LK}. However, we have not found an explicit construction of the corresponding fractal in the literature.

\begin{Thm}\label{T-Main-hex}
A point $ (x, y) \in Q $ lies in the $b$-balanced $m$-hexagon fractal $ \mathcal{H}_m^b $ if and only if the base $ m $ representations of $ x $ and $ y $ satisfy
\begin{equation}\label{e-bal-id}
-b \leq [x_i]_{m_b} + [y_i]_{m_b} \leq m - 1 - b
\end{equation}
for every $ i \in \mathbb{N} $.
\end{Thm}

\begin{proof}
%Letting Let $x=[0.x_1x_2\cdots]_m^b$ and $y=[0.y_1y_2\cdots]_m^b$, the identity \eqref{e-bal-id]} is equivalent to   $-b\leq x_j+y_j\leq m-1-b$ for every $i\in\N$.
%
Arguing as in the proof of  Theorem \ref{T-Main-m}, we can easily verify 
the set $ H_m^b(1) $  consists of points $(x,y)$    for which $  -b\leq  x_1+y_1\leq m-1-b$  and contains  all points $ (x,y)$  for which $
x = [0.x_1 00\cdots]_m, \qquad y = [0.y_1 00\cdots]_m,
$ and $
0 \leq [x_1]_m + [y_1]_m \leq m - 1.
$ The rest of the proof  is a line-by-line generalization of the proof of Theorem \ref{T-Main-m}.
\end{proof}

%\begin{tikzpicture}
%	\draw [black, thin]

%	(2,2) -- (-2,2) -- (-2,-2 ) -- (2,  -2) -- cycle;
%	\draw [ fill=gray]   (2, -.5) --(2, .5) --  (.5, .5) --(.5, 2)-- (-2,  2) --(-2, -0.5)--(-0.5, -0.5)-- (-0.5, -2)--( 2,-2)--cycle;
%	\draw[black, ->] (- 2.5,0) -- (2.5,0);
%	\draw[black, ->]  (0,-2.5 ) -- (0, 2.5);
%\end{tikzpicture}
%
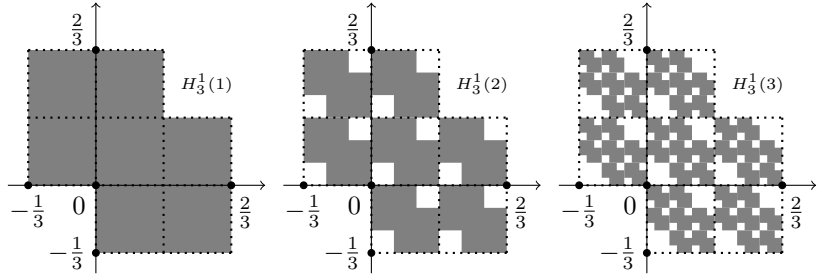
\begin{figure}
\begin{tikzpicture} [scale=0.9]

\newcommand{\myshape}{
	\fill[gray] (1,0)-- (3,0) -- (3,2) -- (2,2) -- (2,3)-- (0,3) --(0,1)--(1,1)-- cycle;
}

\begin {scope} 
\myshape
\end{scope}	
\draw[black, ->] (-0.3,1) -- (3.5,1);
\draw[black, ->] (1,-.3) -- (1, 3.7);
\draw[black, thick, dotted] (0,1) -- (0,3);
\draw[black, thick, dotted] (1,0) -- (1,3);
\draw[black, thick, dotted] (2,0) -- (2,3);
\draw[black, thick, dotted] (3,0) -- (3,2);
\draw[black, thick, dotted] (0,1) -- (3,1);
\draw[black, thick, dotted] (0,2) -- (3,2);
\draw[black, thick, dotted] (0,3) -- (2,3);
\draw[black, thick, dotted] (1,0) -- (3,0);

\draw (0, 1) node[black, below ]{ \small{$-\frac 13$} } ; 
\draw (1, .7) node[black, left]{ \small{$0$} };
%\draw (2,0) node[black, below]{ \small{ $\frac 13$} };
\draw (2.8,.6) node[black, right]{ \small{ $\frac 23$} };
\draw (1,3.3) node[black, left]{ \small{ $ \frac 23$} };
\draw (1,0) node[black, left]{ \small{ $-\frac 13$} };
\draw (2 , 2.5) node[black, right]{ \tiny{ $H_3^1(1)$} };

\draw[fill=black]  (3,1)  circle(0.05);
%\draw[fill=black]  (2,1)  circle(0.05);
\draw[fill=black]  (1,1)  circle(0.05);
\draw[fill=black]  (0,1)  circle(0.05);
\draw[fill=black]  (1,3)  circle(0.05);
%\draw[fill=black]  (1,2)  circle(0.05);
%\draw[fill=black]  (1,1)  circle(0.05);
\draw[fill=black]  (1,0)  circle(0.05);

\bigskip
\end{tikzpicture}
\begin{tikzpicture} [scale=0.9]
%\draw[black, dotted] (2,0) -- (2,6);
%\draw[black, dotted] (6,0) -- (6, 6 );
%\draw[black, dotted] (0,2) -- (6,2);
%\draw[black, dotted] (0,6) -- (6, 6 );

\newcommand{\myshape}{
\fill[gray] (1,0)-- (3,0) -- (3,2) -- (2,2) -- (2,3)-- (0,3) --(0,1)--(1,1)-- cycle;
}
\newcommand{\Myshape }   {
\foreach \x/\y in {1/0, 2/0,  0/1,  1/1, 2/1, 0/2,1/2 } {
	\begin{scope}[shift={(\x,\y)}, scale=1/3]
		\myshape
	\end{scope}
}}

	\begin{scope}%[shift={(\x,\y)}, scale=1/3]
		\Myshape
	\end{scope}
	%}
\draw[black, ->] (-0.3,1) -- (3.5,1);
\draw[black, ->] (1,-.3) -- (1, 3.7);
\draw[black, thick, dotted] (0,1) -- (0,3);
\draw[black, thick, dotted] (1,0) -- (1,3);
\draw[black, thick, dotted] (2,0) -- (2,3);
\draw[black, thick, dotted] (3,0) -- (3,2);
\draw[black, thick, dotted] (0,1) -- (3,1);
\draw[black, thick, dotted] (0,2) -- (3,2);
\draw[black, thick, dotted] (0,3) -- (2,3);
\draw[black, thick, dotted] (1,0) -- (3,0);

\draw (0, 1) node[black, below ]{ \small{$-\frac 13$} } ; 
\draw (1, .7) node[black, left]{ \small{$0$} };
%\draw (2,0) node[black, below]{ \small{ $\frac 13$} };
\draw (2.8,.6) node[black, right]{ \small{ $\frac 23$} };
\draw (1,3.3) node[black, left]{ \small{ $ \frac 23$} };
\draw (1,0) node[black, left]{ \small{ $-\frac 13$} };
\draw (2 , 2.5) node[black, right]{ \tiny{ $H_3^1(2)$} };
\draw[fill=black]  (3,1)  circle(0.05);
%\draw[fill=black]  (2,1)  circle(0.05);
\draw[fill=black]  (1,1)  circle(0.05);
\draw[fill=black]  (0,1)  circle(0.05);
\draw[fill=black]  (1,3)  circle(0.05);
%\draw[fill=black]  (1,2)  circle(0.05);
%\draw[fill=black]  (1,1)  circle(0.05);
\draw[fill=black]  (1,0)  circle(0.05);
\end{tikzpicture}
\begin{tikzpicture}[scale=0.9]
%\draw[black, dotted] (2,0) -- (2,6);
%\draw[black, dotted] (6,0) -- (6, 6 );
%\draw[black, dotted] (0,2) -- (6,2);
%\draw[black, dotted] (0,6) -- (6, 6 );

\newcommand{\myshape}{
	\fill[gray] (1,0)-- (3,0) -- (3,2) -- (2,2) -- (2,3)-- (0,3) --(0,1)--(1,1)-- cycle;
}
\newcommand{\Myshape }   {
	\foreach \x/\y in {1/0, 2/0,  0/1,  1/1, 2/1, 0/2,1/2 } {
		\begin{scope}[shift={(\x,\y)}, scale=1/3]
			\myshape
		\end{scope}
}}

\foreach \x/\y in {1/0, 2/0,  0/1,  1/1, 2/1, 0/2,1/2 } {
	\begin{scope}[shift={(\x,\y)}, scale=1/3]
		\Myshape
\end{scope}}
\draw[black, ->] (-0.3,1) -- (3.5,1);
\draw[black, ->] (1,-.3) -- (1, 3.7);
\draw[black, thick, dotted] (0,1) -- (0,3);
\draw[black, thick, dotted] (1,0) -- (1,3);
\draw[black, thick, dotted] (2,0) -- (2,3);
\draw[black, thick, dotted] (3,0) -- (3,2);
\draw[black, thick, dotted] (0,1) -- (3,1);
\draw[black, thick, dotted] (0,2) -- (3,2);
\draw[black, thick, dotted] (0,3) -- (2,3);
\draw[black, thick, dotted] (1,0) -- (3,0);

\draw (0, 1) node[black, below ]{ \small{$-\frac 13$} } ; 
\draw (1, .7) node[black, left]{ \small{$0$} };
%\draw (2,0) node[black, below]{ \small{ $\frac 13$} };
\draw (2.8,.6) node[black, right]{ \small{ $\frac 23$} };
\draw (1,3.3) node[black, left]{ \small{ $ \frac 23$} };
\draw (1,0) node[black, left]{ \small{ $-\frac 13$} };
\draw (2 , 2.5) node[black, right]{ \tiny{ $H_3^1(3)$} };

\draw[fill=black]  (3,1)  circle(0.05);
%\draw[fill=black]  (2,1)  circle(0.05);
\draw[fill=black]  (1,1)  circle(0.05);
\draw[fill=black]  (0,1)  circle(0.05);
\draw[fill=black]  (1,3)  circle(0.05);
%\draw[fill=black]  (1,2)  circle(0.05);
%\draw[fill=black]  (1,1)  circle(0.05);
\draw[fill=black]  (1,0)  circle(0.05);
\end{tikzpicture}
\caption{ \small The first steps of the construction of the hexagon fractal ${\cal H}_3^1$}
\end{figure}
\bigskip

\medskip

\section{The dimension of the triangle and the hexagon fractals }

The Minkowski (or box-counting) dimension can be used to quantify  the  size  of a fractal set in terms of its scaling properties. Given a bounded set   $A\subset \mathbb{R}^n$, we consider the number $N(\varepsilon)$ of cubes of side   $\varepsilon > 0$ needed to cover the set. The \emph{Minkowski dimension} of $A$ is defined (when the limit exists) as
$$
\dim_M(A) := \lim_{\varepsilon \to 0^+} \frac{\log N(\varepsilon)}{\log(1/\varepsilon)}.
$$
The Minkowski dimension can be interpreted as the rate at which the number of boxes needed to cover the set grows as the box size decreases.  It  closely related to the Hausdorff dimension   though generally easier to compute and less sensitive to irregularities. %For many well-behaved fractal sets, the two notions coincide.

For    the  definitions and a detailed  treatment of   the fractal measures and the Hausdorff  and Minkowski dimensions   see e.g \cite[Chapt.1]{E} and also   \cite{F, tricot}. For   applications of the Minkowski dimension  in dynamics and physics, see \cite{mandelbrot}.

Our construction of the  triangle and hexagon fractals allows an easy calculation of their Minkowski dimension.

\begin{Thm}\label{T-dim}
a) For every integer $m\ge 2$, the Minkowski dimension of the triangle fractal ${\mathcal T}_m$ is \begin{equation}\label{e-dim-T} d_M({\mathcal T}_m)=\frac{\log(\frac{m(m+1)}{2} )}{\log(m)}.\end{equation}

b) For every  integers $m>2$ and $b\in [1, \frac m2]$, the  Minkowski  dimension of the hexagon fractal $ {\mathcal H}_m^b$ is  \begin{equation}\label{e-dim-H} d_M({\mathcal H}_m^b)=\frac{\log\big(\frac{m(m+1)}{2}+b(m-1-b)\big)}{\log(m)}.\end{equation}

\end{Thm}

\medskip

To  prove Theorem \ref{T-dim} we need the following 
\begin{Lemma}\label{L-lattice}
For every integer    $m\ge 2$  and $b\in [0, \frac m2]$,  let 
$$L_{m}^b:=\{(k, h)\in \Z\times \Z \ :  \  h,\, k,\, h+k\in [-b, \  m-1-b]   \}.$$
The cardinality of $L_{m}^b$  is \begin{equation}\label{e-lmb}  \ell_m^b :=\dsize\frac{m(m-1)}{2}+b(m-1-b). \end{equation}    

\end{Lemma}	
%We  have denoted with $|A|$ the cardinality of a finite set $A$.
\begin{proof}
When $b=0$, the points of   $L_{m}^b$  lie in   the closed  triangle with  vertices $(0,0)$, $(0, m-1)$ and $(m-1,0)$. Thus, $L_m^b$ contains 
$m+ (m-1)+\cdots +1 = \frac{m(m+1)}{2} $ points,   and \eqref{e-lmb} holds in this special case.

\medskip
Assume now $m>2$ and $b\ge 1$;  
fix $h^*\in \{0,\, ... \, m-b-1\}$.    The condition $-b\leq h^*+k\leq m-1-b $ yields 
$-b-h^*\leq k\leq m-1-b-h^*$, but  since we require $k\in [-b, m-1-b]$, we can only have 
$ -b\leq k\leq m-1-b-h^*$.   Thus, the  set  $L_m^b$ contains $m-1-h^* $ points  with coordinates  $(h^*,\, k)$,  and  
$$\sum_{h^*=0}^{m-b-1} m -1-h^*= (m-1)(m-b ) -\frac{(m-b-1)(m-b )}{2}
$$ points  $(h,k)$ with $h\ge 0$.

When $h^*\in \{-b, \, -b+1\, \cdots -1 \}$,   $k$ satisfies the inequality  
$-b-h^*\leq k\leq m-1-b  $.  Thus, the set $L_b$ contains $m-1+h^* $ points  with coordinates  $(h^*,\, k)$   and 
$$\sum_{h^*=1}^{b} (m -1-h^*) %= \sum_{ h^*=1}^{b} (m -1-h^*) 
= (m-1)b -\frac{b(b+1)}{2}
$$ points  $(h,k)$ with $h\in [-b, -1]$.

We can conclude that  the cardinality of  $L_m^b $ is
\begin{align*}
(m-1)(m-b ) &-\frac{(m-b-1)(m-b )}{2} +(m-1)b -\frac{b(b+1)}{2} \\ &=\frac{m(m-1)}{2}+b(m-1-b),  
\end{align*}
as required.
\end{proof}

\begin{proof}[Proof of Theorem \ref{T-dim}]

The triangle  fractal ${\mathcal T}_m$  is the intersection of  the decreasing family of sets $\{T_m(n)\}_{n\in\N}$ defined in \eqref{e-def-Tj}. Our construction shows that $T_m(n)$ is  the union of $(\ell_m^0)^n= (\frac{m(m+1)}{2})^n $ squares of side $m^{-n}$. 
For every $n\in\N$, we have that 
$  
\frac{\log ((\ell_m^0)^n)}{\log(m^n)}= \frac{\log(\frac{m(m+1)}{2})}{\log(m)},
$ 
and so $$ d_M({\mathcal T}_m)= \lim_{n\to\infty} \frac{\log ((\ell_m^0)^n)}{\log(m^n)}= \frac{\log(\frac{m(m+1)}{2})}{\log(m)}$$ as  required.

The  hexagon   fractal ${\mathcal H}_m^b$  is the intersection of  the   sets $\{H_m^b(n)\}_{n\in\N}$ defined in \eqref{e-Hn},  each of which  consists of  $(\ell_m^b)^n $ squares of side $m^{-n}$.   Thus, the Minkowski dimension  of the  hexagon fractal  ${\mathcal H}_m^b$ is  
$$
d_M({\mathcal H}_m^b)=\lim_{n\to\infty}  \frac{\log ((\ell_m^b)^n)}{\log(m^n)}=\frac{\log(\frac{m(m+1)}{2}+b(m-1-b))}{\log(m)}
$$
as required.
\end{proof}

When $m=2$,   Theorem \ref{T-dim}  shows that   the dimension of the Sierpiński triangle  $S_T$ is $\frac{\log(3)}{\log(2)}$, a well-known result.

\medskip
We conclude the paper with the following observation: 
The argument used in proof of Theorem~\ref{T-dim} can be used to show that   the  Lebesgue measure of the sets $T_m(n)$  is  
$
|T_m(n)| = m^{-2n} \left(\frac{m(m+1)}{2}\right)^n.
$
Since $\mathcal{T}_m=\bigcap_{n=0}^\infty T_m(n)$ and 
$\dsize 
\lim_{n\to\infty} |T_m(n)| = \lim_{n\to\infty} \left(\frac{1}{2} + \frac{1}{2m}\right)^n = 0,
$
it follows that  the triangle fractal $\mathcal{T}_m$ has Lebesgue measure zero. A similar argument shows that the Lebesgue measure of the hexagon  fractals $ \mathcal{H}_m^b $ is also zero.

However,    the sets $\mathcal{T}_m$ and $ \mathcal{H}_m^b $  become increasingly dense     as $m\to\infty$ because, using  for example, L'Hôpital's Rule, we can verify  that
\[
\lim_{m \to \infty} d_M(\mathcal{T}_m) = \lim_{m \to \infty} d_M(\mathcal{H}_m^b) = 2.
\]

This means that although the triangle and the hexagon fractals  have measure zero,   they asymptotically behave like two-dimensional objects when $m$   (which can be viewed as a   measure of  their fractal complexity) goes to infinity.

\bigskip
\noindent
ACKNOWLEDGMENTS:  All authors of this paper were equally involved in the conception and design  of the paper, its final drafting, and  its critical revision for intellectual content. All authors agree to be accountable for all aspects of the work.  

We  have no competing interests to declare.

No funding was received


\begin{thebibliography}{99}



\bibitem{BP} J.-P. Bouchaud and M. Potters, \textit{Theory of Financial Risk and Derivative Pricing: From Statistical Physics to Risk Management}, Cambridge Univ. Press, 2003.

\bibitem{C} E. Conversano and E. Tedeschini Lalli, \textit{Sierpiński triangles in stone, on medieval floors in Rome}, \textit{J. Appl. Math.} \textbf{4} (2011), no. 4.

\bibitem{E} G. Edgar, \textit{Measure, Topology, and Fractal Geometry}, 2nd ed., Springer, 2008.

\bibitem{F} K. J. Falconer, \textit{Fractal Geometry: Mathematical Foundations and Applications}, 3rd ed., Wiley, 2014.

\bibitem{FT} E. Faghih, M. Taheri, K. Navi, and N. Bagherzadeh, \textit{Efficient realization of quantum balanced ternary reversible multiplier building blocks: A great step towards sustainable computing}, \textit{Sustainable Comput. Inform. Syst.} \textbf{40} (2023).

\bibitem{GKP} R. L. Graham, D. E. Knuth, and O. Patashnik, \textit{Concrete Mathematics: A Foundation for Computer Science}, 2nd ed., Addison-Wesley, 1994.

\bibitem{LK} J. Leskovec and C. Faloutsos, \textit{Scalable modeling of real graphs using Kronecker multiplication}, in \textit{Proc. 24th Int. Conf. Machine Learning}, ACM, 2007, 497--504.

\bibitem{mandelbrot} B. Mandelbrot, \textit{The Fractal Geometry of Nature}, W. H. Freeman, 1982.

\bibitem{MW} J. Moran and K. Williams, \textit{From Tiling the Plane to Paving Town Square}, 2005, DOI: 10.1007/3-540-26443-4\_2.

\bibitem{P} E. Pickover, \textit{Computers, Pattern, Chaos and Beauty: Graphics from an Unseen World}, St. Martin’s Press, 1990.

\bibitem{S} M. Schroeder, \textit{Fractals, Chaos, Power Laws: Minutes from an Infinite Paradise}, Dover, 2009.

\bibitem{Si} W. Sierpiński, \textit{Sur une courbe dont tout point est un point de ramification}, \textit{C. R. Acad. Sci. Paris} \textbf{160} (1915), 302--305.

\bibitem{TT2} M. Toulabinejad, M. Taheri, K. Navi, and N. Bagherzadeh, \textit{Toward efficient implementation of basic balanced ternary arithmetic operations in CNFET technology}, \textit{Microelectron. J.} \textbf{90} (2019), 267--277.

\bibitem{tricot} C. Tricot, \textit{Curves and Fractal Dimension}, Springer, 1995.

\bibitem{W} P. West, \textit{Fractal Physiology and Chaos in Medicine}, World Scientific, 2013.

\end{thebibliography}
\end{document}